\documentclass[10pt]{amsart}
\usepackage{amsmath,amsfonts,latexsym,amssymb,amscd,
  graphicx,amsthm, mathtools}
\usepackage{tikz, subcaption}
\usetikzlibrary{decorations.pathreplacing, calc, patterns}
\usepackage{enumerate}
\usepackage{mdwlist}
\usepackage{ulem}
\usepackage{cancel}
\usepackage{hyperref}

\allowdisplaybreaks[1]

\newcommand{\ONE}{{\mathbf{1}}}

\newcommand{\N}{{\mathbb N}}

\newcommand{\Fc}{\mathcal{F}}

\newcommand{\Pp}{\mathsf{P}}
\newcommand{\Z}{{\mathbb Z}}
\newcommand{\iid}{i.i.d.\ }

\newcommand{\E}{\mathsf{E}\,}

\newcommand{\R}{{\mathbb R}}

\renewcommand{\Mc}{{\mathcal M}}

\newcommand{\eps}{{\varepsilon}}

\newcommand{\argmax}{\mathop{\mathrm{argmax}}}

\newtheorem{theorem}{Theorem}[section]

\newtheorem{lemma}{Lemma}[section]

\numberwithin{equation}{section}

\renewenvironment{proof}[1][Proof]{% beginning of proof
{\noindent {\sc #1: }}
}{% end of proof
{{\hfill $\Box$ \smallskip}}
}
\makeatletter
\let\orgdescriptionlabel\descriptionlabel
\renewcommand*{\descriptionlabel}[1]{%
  \let\orglabel\label
  \let\label\@gobble
  \phantomsection
  \edef\@currentlabel{#1}%
  \let\label\orglabel
  \orgdescriptionlabel{#1}%
}
\makeatother

\title[]{Strongly mixing smooth planar vector
  field without asymptotic directions}
\author{Yuri Bakhtin}
\address{Courant Institute of Mathematical Sciences\\ New York University \\ 251 Mercer St, New York, NY 10012 }
\author{Liying Li}
\email{bakhtin@cims.nyu.edu, liying@cims.nyu.edu}
\begin{document}

\begin{abstract}
We use a Voronoi-type tesselation based on a compound Poisson point process to construct a polynomially mixing stationary random smooth planar vector field with bounded nonnegative components such that, with probability one, none of the associated integral curves possess an asymptotic direction.
\end{abstract}

 \maketitle

 \section{Introduction and the main results}\label{sec:motivation} 
 
Let $v$ be a smooth vector field on~$\R^2$ such that for every~$z\in\R^2$, the integral curve~$\gamma_z: \R_+  \to \R^2$  (here $\R_+=[0,\infty)$) is well-defined as a unique solution of the autonomous ODE
\begin{equation}
\label{eq:ode}
\dot \gamma_z(t) = v\big( \gamma_z(t) \big),
\end{equation}
satisfying
\begin{equation}
\label{eq:init-cond-ode}
\gamma_z(0) = z.
\end{equation}

\medskip
 
Being motivated by homogenization  problems for stochastic Hamilton--Jacobi (HJ) type equations (see~\cite{Souganidis:MR1697831},\cite{Rezakhanlou-Tarver:MR1756906},\cite{Nolen-Novikov:MR2815685},\cite{Cardaliaguet-Souganidis:MR3084699},\cite{Jing-Souganidis-Tran:MR3817561}), limit shape problems in First Passage Percolation (FPP) type models (see, e.g., \cite{AHDbook:MR3729447}), and related straightness properties of random optimal paths in random environment (see \cite{Licea1996},\cite{HoNe},\cite{Wu},\cite{CaPi},\cite{Cardaliaguet-Souganidis:MR3084699},\cite{BCK:MR3110798},\cite{kickb:bakhtin2016}),
in \cite{Bakhtin-Li-weakly-mixing}, we raised the problem of conditions on a stationary random smooth vector field $v$ that would guarantee that with probability $1$ the asymptotic direction $\lim_{t\to\infty} (\gamma_z(t)/t)$ is well-defined for all $z\in\R^2$.

A simple argument based on the strong law of large numbers implies that such a straightness statement holds for vector fields $v$ with bounded nonnegative components and finite dependence range. However, it is not clear how much the finite dependence range requirement can be relaxed. In \cite{Bakhtin-Li-weakly-mixing} we constructed an example of a {\it weakly mixing} stationary random $2$-dimensional vector field $v$ 
with nonnegative components such that,
with probability $1$, the following holds for all
$z\in\R^2$:
\begin{equation}
\label{eq:trajects-go-to-infty}
\lim_{t\to \infty}|\gamma_z(t)|=\infty,
\end{equation}
 \begin{equation}\label{eq:integral-curves-no-direction}
  \liminf_{t \to \infty} \frac{\gamma^{2}_z(t)}{  \gamma^{1}_z (t) } = 0, \quad
  \limsup_{t \to \infty} \frac{\gamma^{2}_z(t) }{  \gamma^{1}_z (t)} = \infty.
\end{equation}
In other words, with probability one, none of the integral curves defined by this vector field have an asymptotic direction. Instead, they sweep through a cone of partial asymptotic directions.

The goal of this note is to give an example of a a {\it strongly mixing} vector field with no asymptotic directions.

Before we state our result, let us remark that
the construction in \cite{Bakhtin-Li-weakly-mixing} was based on a modification of the discrete lattice example from~\cite{2016arXiv161200434C}, with similar properties and thus it has only the weak mixing property and not the strong one. 
Moreover, due to the product nature of the construction in \cite{Bakhtin-Li-weakly-mixing} (which means that random transformations are applied independently to both components), applying it
straightforwardly, even to strongly $\Z^2$-mixing lattice fields with similar properties
like that from~\cite{Bramson-Zeitouni-Zerner:MR2243869},
produces random vector fields that are not strongly mixing with respect to the action of $\R^2$. 
The strongly mixing example we give in this note, allows for analysis in the spirit of \cite{Ziliotto:MR3684310}.
Our example is also related to the homogenization problem of the non-convex Hamilton--Jacobi equation studied in \cite{FSHomogenizationNonhomogenizationCertain2017,FFZExampleFailureStochastic2021}.

To state our main result, we denote the two components of $v\in\R^2$ by  $v^1$ and $v^2$.
\begin{theorem}\label{th:no-average-slope}
There is a strongly mixing stationary smooth vector field $v$ on $\R^2$ such that with probability~$1$, for all $z\in\R^2$, 
\begin{equation}
\label{eq:up-right}
v^1(z) , v^2(z)\ge 0, \quad v^1(z) + v^2(z)= 1,
\end{equation}
and identities \eqref{eq:trajects-go-to-infty}, \eqref{eq:integral-curves-no-direction} hold.
\end{theorem}

This theorem means that mixing is not enough to guarantee the asymptotic straightness of integral curves. In 
 Lemma~\ref{lem:strong-mixing}, we  actually
show that a polynomial estimate on the mixing rate holds for our example.  Probably there are stronger conditions on the rate of mixing sufficient for straightness but this question remains open.

Our vector field, similarly to the previous examples from \cite{Bakhtin-Li-weakly-mixing}, \cite{2016arXiv161200434C}, and their FPP predecessor~\cite{Haggstrom-Meester:MR1379157}, traps
the integral curves in long narrow channels each stretched along one of the extreme directions, so that the
curves oscillate between these two directions never settling on any specific one. 

In our new example, the construction of these channels is based on a Voronoi-type tesselation
of the plane with centers of influence at Poissonian points. Each Poissonian point is equipped  with a  rectangular domain of influence, a narrow channel with heavy-tailed random length, and an additional random strength parameter that helps to decide which influence wins in the case of channel overlaps.

We describe our construction and prove the strong mixing property in Section~\ref{sec:construction}.
We study the flow generated by our random vector field in Section~\ref{sec:flow-asymptotics}. 
In Section \ref{sec:discussion}, we give a discussion of our model and its comparison to \cite{Ziliotto:MR3684310}.

{\bf Acknowledgments.} In the first version of the paper, the proof that identities \eqref{eq:trajects-go-to-infty}, \eqref{eq:integral-curves-no-direction} hold for our example was quite involved. We
would like to thank the anonymous referee whose remark stimulated us to find a short efficient argument based on the idea from~\cite{Ziliotto:MR3684310}.  YB is grateful to NSF for partial support via grant DMS-1811444.

\section{Construction and Strong mixing}
\label{sec:construction}
Our construction is based on a Poissonian point field.  Let $(\Omega_0, \Fc_0, \Pp_0)$ be a complete probability space, where $\Omega_0$ is identified as the space of all locally finite configurations
$\omega= \{\eta_i = (x_i,  r_i, \xi_i , \sigma_i), i \in \N\}$ in $\mathcal{X} = \R^2 \times\R \times
\R \times \Sigma$ where~${\Sigma = \{ 1,2 \}}$. 
Configurations $\omega$ are sets, with no canonical enumeration. As usual, we use an arbitrary 
enumeration for convenience.

The~$\sigma$-algebra $\Fc_0$ is
generated by all the maps~$\omega \mapsto n(\omega \cap B)$,
where $B$ is any bounded Borel set in $\mathcal{X}$ and $n(\cdot)$ counts the
number of points in a set.
The measure~$\Pp_0$ is the distribution of a Poisson point field with the following intensity~$\mu$:
\begin{equation}
  \label{eq:intensity-mu}
  \mu(dx \times d r  \times d\xi  \times d\sigma)
  =\frac{1}{2}    \frac{\alpha e^{-r}}{\xi^{\alpha+1}} \ONE_{\{r \ge 0,\ \xi \ge 1\}}\, dx\, dr\, d\xi\, d \sigma
  := f(x, \sigma, r, \xi)\, dx\, dr\,d\xi\, d \sigma.
\end{equation}
where~$1 < \alpha \le 2$ is a fixed number, and on the right hand side $dx, dr, d\xi$ are the Lebesgue measure and
$d\sigma$  is the counting measure.
Since $\mu$ has no atoms when projected onto the $x$-component or
$\xi$-component, we see that with probability one,
\begin{equation}
\label{eq:3}
x_i \neq x_j, \ \xi_i \neq \xi_j, \quad i \neq j.
\end{equation}
This allows us to work on a modified probability space~$\Omega$ with full measure:
\begin{equation*}
\Omega = \{  \omega: \text{ (\ref{eq:3}) holds true} \}.
\end{equation*}
Let us denote by $\Fc$ and $\Pp$ the restriction of $\Fc_0$ and $\Pp_0$ onto $\Omega$.  From now on
we will work with the probability space $(\Omega, \Fc, \Pp)$.
We will also denote the components of~$\eta \in \mathcal{X}$ by~$x(\eta)$, $\xi(\eta)$, etc.
We can interpret this Poisson point field as a compound Poisson point process in the usual way: the spatial footprints $x_i$ form
a homogeneous Poisson point process in $\R^2$ with Lebesgue intensity; each $x_i$ is equipped with 
labels $r_i,\xi_i,\sigma_i$ that are mutually independent and independent of labels of other points, with distributions~$\mathrm{Exp}(1)$, $\mathrm{Par}(\alpha)$, and uniform on $\Sigma$. Here we denote by~$\mathrm{Exp}(\lambda)$ the exponential distribution with parameter~$\lambda>0$,
with Lebesgue density~$\lambda e^{-\lambda r}\ONE_{\{r\ge 0\}}$, and by
~$\mathrm{Par}(\alpha)$ the Pareto distribution with parameter~$\alpha$, with density~$\frac{a}{t^{a+1}}\ONE_{\{ t\ge 1 \}}$.
We refer to \cite[Section 6]{daley_introduction_2003} for the background on compound Poisson processes.

In the rest of the section we will construct a random vector field given any
fixed configuration $\omega$.
Let $e_1, e_2$ be the standard basis in~$\R^2$.  We often write~${x=(x^1,x^2)}$ for a point in~$\R^2$.
For each~$\eta_i \in \omega$, let us associate with $x_i$ a \textit{domain of influence}~$D_i$, which is a rectangle of length
$r_i \xi_i$ and width $1$ in the direction of~$e_{\sigma_i}$.  More precisely, 
we define 
\begin{align*}
  \mathrm{D}: \mathcal{X} &\longrightarrow \text{ rectangles in $\R^2$, }\\
  \eta = (x^1, x^2, r, \xi , \sigma) & \longmapsto   
\begin{cases}
[x^1, x^1 + r\xi] \times [x^2, x^2+1], &  \sigma = 1,  \\
[x^1, x^1 + 1] \times [x^2, x^2 + r\xi],   & \sigma = 2.  \\
\end{cases}
\end{align*}
and let $D_i = \mathrm{D}(\eta_i)$. We call $\eta_i$ the base point and~$\xi_i$ the strength of the
domain~$D_i$.
For any region~$R \subset \R^2$, we also define~$\mathrm{D}^{-1}(R) \subset \mathcal{X}$ as
\begin{equation*}
\mathrm{D}^{-1}(R) = \{  \eta \in \mathcal{X}: \mathrm{D}(\eta) \cap R \neq \varnothing\}.
\end{equation*}
\begin{lemma}
\label{lem:finite-number-of-influence-region}
With probability one, every bounded set in $\R^2$ intersects with a finite number of domains of
influence.
\end{lemma}
\begin{proof}
  It suffices to show that for all~$m,n  \in \Z$, with probability one the unit square~${R = [m,m+1]}
  {\times [n,n+1]}$
  intersects with a finite number of $D_i$'s.  This is equivalent to 
$\mu( \mathrm{D}^{-1}(R)) < \infty$.
Without loss of generality let us assume~$R = [0,1]^2$.  We have 
\begin{multline*}
\mathrm{D}^{-1}(R) = \{\eta = (x^1, x^2, r, \xi, \sigma): \sigma = 2,\ x^2 \le 1,\ -1 \le x^1 \le 1,\
0 \le x^2 + r\xi  \} \\
\cup
\{\eta = (x^1, x^2, r, \xi, \sigma): \sigma = 1,\ x^1 \le 1,\ -1 \le x^2 \le 1,\ 0 \le x^1 + r\xi  \}
\end{multline*}
and
\begin{align*}
\mu ( \mathrm{D}^{-1}(R)) &= 2\int_{\{ \sigma = 2,\, x^2 \le 1,\, -1 \le x^1 \le 1,\, 0 \le x^2 + r\xi  \}}
                            f(x,r,\xi,\sigma) dx\,dr\, d\xi\, d\sigma \\
                          &= \int_{-1}^1 dx^1 \int_{-\infty}^1 dx^2 \int_1^{\infty} \frac{\alpha}{ \xi^{\alpha+1}}  \, d\xi
                            \int_{(-x^2)_+/\xi}^{+\infty} e^{-r} \, dr \\
  &= 2 + 2 \int_{-\infty}^0dx^2 \int_1^{\infty} \frac{\alpha}{ \xi^{\alpha+1}}  \, d\xi \cdot
    e^{\frac{x^2}{\xi}} \\
  &= 2 + 2 \int_1^{\infty} \frac{\alpha}{ \xi^{\alpha}} \, d\xi < \infty,
\end{align*}
where we used $\int_{-\infty}^{1}=\int_{-\infty}^{0}+\int_{0}^{1}$ in the third line, and~$\alpha > 1$
in the last line.  
\end{proof}

For~$\Lambda \subset \mathcal{X}$, we denote by $\mathcal{F}_{\Lambda}$ the $\sigma$-algebra generated by all the
maps~$\omega \mapsto n(\omega \cap B )$, where~$B \subset \Lambda$ is any bounded Borel set.
Let~$\Theta$ be a special element and for~$\mu(\Lambda) < \infty$ we define~$\phi(\Lambda) \in \mathcal{X} \cup \{ \Theta \}$ as
\begin{equation*}
\phi(\Lambda) = 
\begin{cases}
\Theta, & \Lambda \cap \omega = \varnothing,  \\
\argmax \{ \xi(\eta) :  \eta \in \Lambda \cap \omega \},   &  \Lambda \cap \omega \neq \varnothing.
\end{cases} 
\end{equation*}
In other words, when there is at least one Poisson point in~$\Lambda$, $\phi(\Lambda)$ gives the one with highest strength.
For convenience we also assign a strength to the special element~$\Theta$ by setting~$\xi(\Theta) = 0$.
It is clear that~$\phi(\Lambda)$ is measurable with respect to~$\Fc_{\Lambda}$.
For~$x \in \R^2$, we also abuse the notation to write~
\[\phi(x) := \phi(\mathrm{D}^{-1}(\{ x \})).\]
The meaning of~$\phi$ should be clear from the context.

Let~$\rho$ be a smooth probability density supported on $[-1/3,0]^2$.
The desired vector field is constructed as a convolution~$ v = \rho* \tilde{v}$, where
\begin{equation*}
\tilde{v}(x) = 
\begin{cases}
  e_{\sigma(\phi(x))}, & \phi(x) \neq \Theta\\
  \frac{1}{2} (e_1 + e_2), & \phi(x) = \Theta.
\end{cases}
\end{equation*}
Clearly, $\tilde{v}$ satisfies~(\ref{eq:up-right}) with~$v$ replaced
by~$\tilde{v}$.  Therefore, $v=\rho*\tilde{v}$ also satisfies~(\ref{eq:up-right}).
In the rest of this section we will 
state and prove the strong mixing property of $v$, along with a polynomial mixing rate.

For~$z \in \R^2$, let us define the shift operator~$\tilde{\theta}^z$ acting on~$\mathcal{X}$ by
\begin{equation*}
\tilde{\theta}^z (x,r,\xi,\sigma) = (x-z, r, \xi, \sigma).
\end{equation*}
This induces the shift operator~$ \theta^z \omega = \theta^z \{ \eta_i\}: = \{
\tilde{\theta}^z \eta_i \}$ defined on~$\Omega$.
Since $(\tilde{\theta}^z)_{z \in \R^2}$ preserves the measure~$\mu$,
$\{\theta^z\}_{z \in \R^2}$ is a measure-preserving $\R^2$-action on~${(\Omega,
\Fc,\Pp)}$.

We temporarily write $v(x) = v_{\omega}(x)$ to stress its dependence on the Poisson point configuration.
The map~$V: \omega \mapsto v_{\omega}(\cdot)$ is measurable from~$(\Omega, \Fc)$
to $(\Mc, \mathcal{B}(\Mc))$, where~$\Mc$ is the space of continuous vector
fields on~$\R^2$, and~$\mathcal{B}(\Mc)$ is the Borel~$\sigma$-algebra
induced by the LU metric
\begin{equation*}
d(u,v) = \sum_{n = 1}^{+\infty} \frac{ \| u - v \|_{C([-n,n]^2)} \wedge 1}{2^n},
\end{equation*}
Let $\Pp_{\Mc} = \Pp V^{-1}$ be the push-forward of~$\Pp$.  Since 
~$v_{\omega}( x) = v_{\theta^x \omega}\big( (0,0) \big)$, $\{ \theta^z \}_{z \in \R^2}$ is also
a measure preserving $\R^2$-action on $(\Mc, \mathcal{B}(\Mc), \Pp_{\Mc})$.

We will show that the $\R^2$-system~$(\{  \theta^z \}_{z \in \R^2}, \Mc,
\mathcal{B}(\Mc), \Pp_{\Mc})$ is polynomially mixing.
For $r>0$, we say that the system is \textit{polynomially mixing of order $r$} (see also~\cite{Bramson-Zeitouni-Zerner:MR2243869}\cite{Ziliotto:MR3684310}), 
if for every~$N>0$, 
\begin{equation}
    \label{eq:strong-mixing}
\limsup_{|z|_1 \to \infty} |z|_1^r \cdot  \Big|\Pp_{\Mc} ( A \cap \theta^z B) - \Pp_{\Mc} (A) \Pp_{\Mc}(B) \Big| < \infty, \quad A, B \in \mathcal{B}(\Mc_N),
\end{equation}
where $\Mc_N$ is the space of vector fields restricted to the square $L_N = [-N,N]^2$, and $|z|_1=|z^1| + |z^2|$.

We note that polynomial mixing of any order $r>0$ implies strong mixing.
In fact, since~$\mathcal{B}(\Mc) = \bigvee_{N=1}^{\infty} \mathcal{B}(\Mc_N)$, we can approximate sets in $\mathcal{B}(\Mc)$ from sets in~$\mathcal{B}(\Mc_N)$, so it follows
immediately from (\ref{eq:strong-mixing})
\begin{equation*}
\lim_{|z|_1 \to \infty} | \Pp_{\Mc} ( A \cap \theta^z B)- \Pp_{\Mc} (A) \Pp_{\Mc}(B)| = 0, \quad  A, B \in \mathcal{B}(\Mc).
\end{equation*}

\begin{lemma}
  \label{lem:strong-mixing}
The $\R^2$-system~$(\{  \theta^z \}_{z \in \R^2}, \Mc,
\mathcal{B}(\Mc), \Pp_{\Mc})$ is polynomially mixing of order~$\alpha-1$.
\end{lemma}
\begin{proof}
  We fix~$N>0$ and let~$A, B \in \mathcal{B}(\Mc_N)$.
For every $z\in\R^2$, there are functions $h$ and $g$ such that
\begin{equation*}
  \ONE_A(v_{\omega}) = h( \omega_1, \omega_0), \quad
  \ONE_{\theta^z B} (v_{\omega}) = g (\omega_2, \omega_0), 
\end{equation*}
where~$\omega_i = \omega \cap \Lambda_i$ and
\begin{equation*}
\Lambda_0 = \mathrm{D}^{-1}(L_N) \cap \tilde{\theta}^z \mathrm{D}^{-1}
(L_N), \quad
\Lambda_1 = \mathrm{D}^{-1} (L_N) \setminus \Lambda_0, \quad
\Lambda_2 = \tilde{\theta}^z \mathrm{D}^{-1} (L_N) \setminus \Lambda_0.
\end{equation*}
Here, for simplicity we suppressed the dependence  of~$g$, $h$ and~$\omega_i$'s on~$z$.
Let $\bar{h} (\omega_0) = \E \big[  h(\omega_1, \omega_0) |  \omega_0\big]$ and~$\bar{g} (\omega_0) = \E \big[  g(\omega_2, \omega_0) |  \omega_0\big]$.
By independence of~$\omega_i$'s, 
\begin{align*}
\Pp_{\mathcal{M}}(A \cap \theta^z B) & = \E h(\omega_1, \omega_0) g(\omega_2, \omega_0) = \E \bar{h}(\omega_0)
                                                 \bar{g}(\omega_0) \\
&  = \bar{h}(\varnothing) \bar{g}(\varnothing) \Pp(\omega_0 = \varnothing)+   \E \bar{h}(\omega_0)   \bar{g}(\omega_0)
    \ONE_{\omega_0 \neq \varnothing}.
\end{align*}
Using this and noting that $0 \le \bar{g}, \bar{h} \le 1$, we obtain 
\begin{equation}
\label{eq:13}
\Big|   \Pp_{\mathcal{M}}(A \cap \theta^z B) - \bar{h}(\varnothing) \bar{g}(\varnothing)
\Big| \le  2 \Pp (\omega_0 \neq \varnothing).
\end{equation}
We also have 
\begin{multline*}
  \Pp_{\Mc} (A) \Pp_{\Mc}(B)
  = \E \bar{h}(\omega_0) \E \bar{g}(\omega_0)
  \\
  = \Big( \bar{h}(\varnothing) + \E (\bar{h}(\omega_0 ) - 1 ) \ONE_{\omega_0 \neq \varnothing}
  \Big)
  \Big( \bar{g}(\varnothing) + \E (\bar{g}(\omega_0 ) - 1 ) \ONE_{\omega_0 \neq \varnothing}
  \Big),
\end{multline*}
and therefore 
\begin{equation}\label{eq:11}
\Big|  \Pp_{\Mc} (A) \Pp_{\Mc}(B) - \bar{h}(\varnothing)\bar{g}(\varnothing)  \Big| \le 3 \Pp (
\omega_0 \neq \varnothing).
\end{equation}
So if we show that
\begin{equation}
\label{eq:14}
\limsup_{|z|_1 \to \infty} |z|_1^{\alpha-1}\Pp (\omega_0 \neq \varnothing) < \infty.
\end{equation}
then this and~(\ref{eq:13}), (\ref{eq:11}) will imply (\ref{eq:strong-mixing}).
Let $|z|_1 > 4N$, and without loss of generality assume $z^1\ge z^2>0$.
The limit~(\ref{eq:14}) is equivalent to
\begin{equation*}
\limsup_{z^1 \to \infty}  (z^1)^{\alpha-1}\mu (\Lambda_0) = \limsup_{z^1 \to \infty}  (z^1)^{\alpha-1}\mu \Big(  \mathrm{D}^{-1}(L_N) \cap \tilde{\theta}^z \mathrm{D}^{-1}
(L_N) \Big)  < \infty.
\end{equation*}
Since $\Lambda_0  \subset  \{\eta:  \sigma= 1,\   x^1 < -z^1+N, \  |x^2| \le N+1,\  x^1 + r\xi \ge -N\}$, 
we have 
\begin{align*}
  \mu(\Lambda_0)& \le 2(N+1) \int_{-\infty}^{-  z^1+ N} dx^1 \int_1^{\infty} \frac{\alpha}{ \xi^{\alpha+1}} \, d\xi
                  \int_{\frac{-N-x^1}{\xi}}^{+\infty} e^{-r}\, dr \\
                &= 2(N+1) \int_0^{\infty} dy \int_1^{\infty} \frac{\alpha}{ \xi^{\alpha+1}} \, d\xi e^{-\frac{y+z^1 - 2N}{\xi}}\\
                & = 2(N+1) \int_1^{\infty} \frac{\alpha}{ \xi^{\alpha}} e^{- \frac{z^1 - 2N}{\xi}}\, d\xi \\
                &\le 2(N+1)e^{2N} \int_1^{\infty} \frac{\alpha}{\xi^{\alpha}} e^{-\frac{z^1}{\xi}} \, d\xi  \\
                &= 2(N+1)e^{2N} (z^1)^{1-\alpha} \int_{1/z^1}^{\infty} \frac{\alpha}{\tilde{\xi}^{\alpha}} e^{-\frac{1}{\tilde{\xi}}} \, d\tilde{\xi} \qquad (\xi = z^1\cdot \tilde{\xi}),\\
  & \le (z^1)^{1-\alpha} \cdot 2(N+1)e^{2N} \int_0^{\infty} \frac{\alpha}{\tilde{\xi}^{\alpha}} e^{-\frac{1}{\tilde{\xi}}} \, d\tilde{\xi},
\end{align*}
which implies~(\ref{eq:14}).  
This completes the proof.
\end{proof}

\section{Long-term behavior of integral curves}
\label{sec:flow-asymptotics}

For  $\eps>0$ and $L\ge 1$, let $E_{\eps,L}$ be the event that $\tilde{v} = (1,0)$ on $[0,L] \times [a-1,a]$ for some $1 \le a \le L$.
Since $v = \rho * \tilde{v}$ and the smooth kernel $\rho$ is supported on $[-1/3,0]^2$,
on the event $E_{\eps,L}$, $v \equiv (1,0)$ on $[0,L] \times [a-1/3,a]$, and hence for some $t_0>L$, $\gamma_{(0,0)}(t_0) \in  \{ L \}\times [0,\eps L]$, which implies
$\frac{\gamma^2_{(0,0)}(t_0)}{ \gamma^1_{(0,0)}(t_0)} \le \eps$.
Our proof of Theorem~\ref{th:no-average-slope} is based on the following lower bound of the probability $\Pp (E_{\eps,L})$.
\begin{lemma}
  \label{lem:liminf-E-eps-L}
  For every $\eps>0$, 
\begin{equation*}
\liminf_{L \to \infty} \Pp (E_{\eps,L}) > 0.
\end{equation*}
\end{lemma}

\begin{proof}[Derivation of Theorem~\ref{th:no-average-slope} from Lemma~\ref{lem:liminf-E-eps-L}]
      Due to the construction of~$\tilde v$ and $v$ in Section~\ref{sec:construction}, 
      $\tilde v$
    satisfies~\eqref{eq:up-right} for all $z\in\R$ and hence so does $\tilde{v}$.
    Clearly, $v$ is bounded, $C^{\infty}$-smooth, and
    \eqref{eq:trajects-go-to-infty} holds for all starting points~$z\in\R^2$.
    It remains to check~\eqref{eq:integral-curves-no-direction}

    By Fatou's Lemma and Lemma~\ref{lem:liminf-E-eps-L}, 
\begin{equation*}
\Pp \Big(  \{ E_{\eps,n}, \text{ i.o.\ in}\ n \} \Big) = \E \limsup_{n \to \infty} \ONE_{E_{\eps,n}} \ge \limsup_{n \to \infty} \Pp (E_{\eps,n}) > 0.
\end{equation*}
Hence, there exists an event with positive probability on which $\gamma_{(0,0)}(t_n) \in \{ n \}\times [0,\eps n] $ for infinitely many 
pairs $(n,t_n)\in\N\times\R$ satisfying~$t_n >n$, and hence for $z = (0,0)$,
\begin{equation}\label{eq:2}
\liminf_{t \to \infty} \frac{\gamma^2_{z}(t)}{ \gamma^1_z(t)} \le \eps.
\end{equation}

By ergodicity, with probability $1$,  for every $y \in \R^2$, there exists $z \in \Z^2$ with $z^1 < y^1$ and $z^2>y^2$ such that (\ref{eq:2}) holds for $z$.
Since integral curves do  not intersect, we must have
\begin{equation*}
  \liminf_{t \to \infty} \frac{\gamma_y^2(t)}{\gamma_y^1(t)}
  \le \liminf_{t \to \infty} \frac{\gamma_z^2(t)}{\gamma_z^1(t)} \le \eps.
\end{equation*}
This means that with probability $1$, (\ref{eq:2}) holds for all $z \in \R^2$.  Since $\eps$ is arbitrary, with probability $1$, 
\begin{equation*}
\liminf_{t \to \infty} \frac{\gamma^2_{z}(t)}{ \gamma^1_z(t)} = 0, \quad z \in \R^2.
\end{equation*}
By symmetry, with probability $1$,  
\begin{equation*}
\limsup_{t \to \infty} \frac{\gamma^1_{z}(t)}{ \gamma^2_z(t)} = 0,\quad z \in \R^2.
\end{equation*}
This proves \eqref{eq:integral-curves-no-direction}  and completes the proof of the theorem.
\end{proof}

In the rest of this section we will prove Lemma~\ref{lem:liminf-E-eps-L}.

Let $M=[0,A]\times[0,B]$ be a box and $D = [x^1,x^1+r\xi]\times [x^2,x^2+1]$ be a horizontal domain of influence.
We consider the following two relations between $D$ and $M$:
\begin{itemize}
\item we say that $D$ ``crosses'' $M$ if $D \cap M = [0,A] \times [x^2,x^2+1]$;
\item we say that $D$ ``intersects'' $M$ if $D \cap M \neq \varnothing$.  
\end{itemize}
Similar definitions of crossing and intersection apply to vertical domains of influence as well.
Given a box $M$ and strength~$\zeta\ge 1$,  the numbers of horizontal domains with strength at least~$\zeta$  crossing or intersecting $M$  are Poisson r.v.'s.
We denote their intensities by $I_1(A,B,\zeta)$ and $I_2(A,B,\zeta)$, computed below.

For $D=\mathrm{D}(x^1,x^2,r,\xi,\sigma)$ to cross $M=[0,A]\times[0,B]$ and have strength at least~$\zeta$, we have $x^1 \le 0$, $x^2 \in [0,B-1]$, $\xi \ge \zeta$, $x^2+r\xi \ge A$ and $\sigma=1$. Then
\begin{equation*}
  \begin{split}
    I_1(A,B,\zeta) & = \frac{1}{2}  \int_{\zeta}^{\infty} d\xi\int_0^{B-1} dx^2 \int_{-\infty}^0 dx^1 \int_{x^1+r\xi\ge A} dr \cdot \frac{\alpha e^{-r}}{\xi^{\alpha+1}} \\
    & = \frac{1}{2} (B-1)  \int_{\zeta}^{\infty} \frac{\alpha}{ \xi^{\alpha+1}}\,  d\xi  \int_{-\infty}^0 dx^1  e^{- \frac{A-x^1}{\xi}} \\
    & = \frac{\alpha}{2} (B-1) \int_{\zeta}^{\infty} \xi^{-\alpha} e^{-A\xi^{-1}}\, d\xi\\
    & =  \frac{\alpha}{2} (B-1) G_A(\zeta),
  \end{split}
\end{equation*}
where $G_A(\xi) = \int_{\xi}^{\infty} t^{-\alpha} e^{-A t^{-1}}\,dt$.
If~$A=0$, then the box~$M$ is the line segment $\{ 0 \}\times [0,B]$, and  $G_0(\xi) = \frac{1}{\alpha-1} \xi^{-(\alpha-1)}$.
For $A>0$, by a change of variables, $G_A(\xi) = A^{-(\alpha-1)}G_1(A^{-1}\xi)$ and  $G_1(\xi)\sim \frac{1}{\alpha-1} \xi^{-(\alpha-1)}$ as $\xi \to \infty$.
In summary, we have
\begin{equation}\label{eq:asym-I-1}
I_1(A,B,\zeta) 
\begin{cases}
=  \frac{\alpha}{2(\alpha-1)}(B-1)\zeta^{-(\alpha-1)}, & A=0,  \\
\sim \frac{\alpha}{2(\alpha-1)}(B-1)\zeta^{-(\alpha-1)},   & A>0, \ \zeta/A \to \infty.\\
\end{cases}
\end{equation}

For $D=\mathrm{D}(x^1,x^2,r,\xi,\sigma)$ to intersect $M=[0,A]\times[0,B]$ and have strength at least $\zeta$, there are two cases: if $x^1<0$, then $D$ must cross the line segment
$\{ 0 \} \times [-1,B+1]$, the intensity of this part is given by $I_1(0,B+2,0)$ as computed above; if $x^1>0$, then $D$ intersects $M$ if and only if $(x^1,x^2) \in [0,A]\times[-1,B]$, regardless of $r$.
Hence, 
\begin{equation}\label{eq:I-2}
  \begin{split}
    I_2(A,B,\zeta)&  = I_1(0,B+2,0) + \frac{1}{2} \int_0^A dx^1 \int_{-1}^B dx^2 \int_{\zeta}^{\infty} \frac{\alpha \, d\xi}{\xi^{\alpha+1}} \\
    &= \frac{\alpha}{2(\alpha-1)}(B+1)\zeta^{-(\alpha-1)} + \frac{1}{2}A(B+1)\zeta^{-\alpha}.
  \end{split}
\end{equation}

To estimate $\Pp (E_{\eps,L})$, let us consider the event~$C$ where there exists a horizontal domain of influence with strength at least $L^{(\alpha-1)^{-1}}$ crossing the box $M = [0,L]\times [0,\eps
L]$ (the choice of the threshold $\zeta = L^{(\alpha-1)^{-1}}$  ensures that $\lim_{L \to \infty}I_1(L,\eps L, \zeta) \in (0,\infty)$; although we do not use this explicitly in the proof, it is an important part of the construction).
On $C$, let  $D_0 = \mathrm{D}(X^1,X^2,R,\Xi,1)$ be one of such domains with the highest strength and 
let $F \subset C$ be the event where there is no vertical domain with strength at least $\Xi$ intersecting $D_0$.
Clearly, we have 
\begin{equation*}
\Pp (E_{\eps,L}) \ge \Pp (F) = \Pp (C) \Pp (F|C).
\end{equation*}

Conditioned on $C$, the probability of $F$ does not depend on $R$ and $X^1$, since we are considering the intersection with the box $[X^2,X^2+1] \times [0,L]$.
By translational invariance in the $x^2$-coordinate, and the independence between $X^2$ and the vertical domains, the conditional probability does not depend on $X^2$.
Therefore, the conditional probability depends on $\Xi$ only; in fact, it equals $e^{-I_2(1,L,\Xi)}$, i.e., the probability that the number of vertical domains with strength at least~$\Xi$ intersecting $[0,L]\times [0,1]$ is zero.
Hence, we can write
\begin{equation}\label{eq:prob-F}
 \Pp (E_{\eps,L})\ge  \Pp (C) \int_{L^{(\alpha-1)^{-1}}}^{\infty} d\xi \, q(\xi) e^{-I_2(1,L,\xi)},
\end{equation} 
where $q(\cdot)$ is the conditional density of $\Xi$ on $C$.

We have 
\begin{equation}
  \begin{split}
      q(\xi) &=- \frac{d}{d\xi}\Big[ \frac{\Pp (\Xi > \xi)
  }{ \Pp (C)}
  \Big]
=   - \frac{d}{d\xi}\Big[   \frac{1 - e^{-I_1(L, \eps L, \xi)}
  }{ \Pp (C)}
  \Big]\\
&  \ge  [\Pp(C)]^{-1} c\eps L \xi^{-\alpha} e^{-L \xi^{-1}} e^{ - I_1(L,\eps L, \xi)}.
  \end{split}\label{eq:q}
\end{equation}
Here and below, $c$, $c_i$ stand for some constants depending on~$\alpha$, $\eps$ but not on $L\ge 1$, and may change from place to place.
Combining~(\ref{eq:I-2}), (\ref{eq:prob-F}), (\ref{eq:q}), we have
\begin{equation}\label{eq:3}
    \Pp (E_{\eps,L})   \ge      c_1\eps L \int_{L^{(\alpha-1)^{-1}}}^{\infty} \xi^{-\alpha} e^{-L \big( c_2\xi^{-1} +c_3\xi^{-(\alpha-1)} \big) - I_1(L, \eps L, \xi)} \, d\xi.
  \end{equation}
  If $\alpha<2$, then $L^{(\alpha-1)^{-1}}/L \to \infty$, and by~(\ref{eq:asym-I-1}), $I_1(L, \eps L, \xi) \le c L\xi^{-(\alpha-1)}$, so 
  the right-hand side of (\ref{eq:3}) is at least
\begin{equation*}
\begin{split}
     c_1 \eps L \int_{L^{(\alpha-1)^{-1}}}^{\infty} \xi^{-\alpha} e^{-L \big(  c_2\xi^{-1}  + c_3\xi^{-(\alpha-1)} \big)} \, d\xi 
     &     = c_1\eps  \int_1^{\infty} \zeta^{-\alpha} e^{- c_2L^{\frac{\alpha-2}{\alpha-1}} \zeta^{-1}  - c_3 \zeta^{-(\alpha-1)}} \, d\zeta,\\
     & \to  c_1\eps  \int_1^{\infty} \zeta^{-\alpha} e^{ - c_3 \zeta^{-(\alpha-1)}} \, d\zeta > 0, \quad L \to \infty,
\end{split}
\end{equation*}
where we used the change of variables $\xi = L^{(\alpha-1)^{-1}} \zeta$.
If $\alpha=2$, then $I_1(L,\eps L,\xi) \le c G_1(L^{-1}\xi)$.  With the change of variables $\xi=L\zeta$ the right hand side of (\ref{eq:3}) is at least
\begin{equation*}
c_1\eps  \int_1^{\infty} \zeta^{-\alpha} e^{- (c_2+c_3) \zeta^{-1}    - c_4 \eps G_1(\zeta)} \, d\zeta > 0.
\end{equation*}
This completes the proof of Lemma~\ref{lem:liminf-E-eps-L}.

\section{Discussion}\label{sec:discussion}
In this section, we compare our construction with the example in \cite{Ziliotto:MR3684310}, and discuss our choice of the exponential distribution for the length variable $r$ and the Pareto
distribution for the strength variable $\xi$.

Both \cite{Ziliotto:MR3684310} and our construction place long horizontal and vertical corridors on the plane, whose length is proportional to the strength.
The length in \cite{Ziliotto:MR3684310} also has a heavy tail distribution, as seen from the following computation:
corridors of length $10T_k$ has strength $T_k=4^k$, which occurs with intensity $T^{-2}_k$ on the $\Z^2$-lattice, so 
\begin{equation*}
\Pp (\mathrm{length} \ge 10 T_k) \approx c \sum_{k' \ge k} T_{k'}^{-2} \approx cT_k^{-2}.
\end{equation*}
This corresponds to $\alpha=2$ in our construction.

We extend the range of $\alpha$ to $(1,2]$, which is sharp in the following sense.
First, $\alpha>1$ is needed for the strong mixing (Lemma~\ref{lem:strong-mixing}). Second, following the computation in Section~\ref{sec:flow-asymptotics}, when $\alpha>2$, one
has $\lim_{n \to \infty} I_1(L,\eps L, 1) = 0$ and hence $\Pp(C)=0$, so it is impossible to have a long corridor of length $L$ that is $\eps L$ close to the origin, and the argument in the present
paper and \cite{Ziliotto:MR3684310} will fail.

Another difference between our construction and \cite{Ziliotto:MR3684310} is that the ratio between the length and strength is not a constant $10$, but is given by independent exponential variables.
In addition to simplifying some computations, another motivation to introduce the exponential variable is to use its memoryless effect to establish the following kind of Markov property from the
point of view of trajectories.
Roughly speaking, let $A$ be a Borel subset of $\R_{\ge 0}^2$; then for a deterministic functional on $\mathcal{B}(\R^2) \times \R_{\ge 1}$
\begin{equation}\label{eq:4}
\Pp \Big( \gamma_z (t + \cdot) - \gamma_z(t) \in A\, \big|\  \gamma_z(t) = x, \phi(x) = \xi \Big) = F(A, \xi).
\end{equation}
The time~$t$ can also be replaced by a suitably defined 2-dimensional ``stopping time'' with respect to an appropriate filtration.

Using (\ref{eq:4}), for $\alpha<2$, we obtained an interesting picture for a typical integral curve $\gamma_z$ in an earlier version of this paper written independently of~\cite{Ziliotto:MR3684310}.
There exists $T_z>0$ such that when $t>T_z$, $\phi \big( \gamma_z(t) \big)$ takes values in an increasing sequence $\xi_1$, $\xi_2$, $\cdots$.  The integral curve stays in each long corridor with strength $\xi_i$ for an amount of time proportional to $\xi_i$, and $(\xi_i)$ forms a Markov chain with explicit transition density.
Another intuition to choose heavy tail distribution for the strength is that the sum of \iid  heavy tail random variables is dominated by their maximum (\cite{DarInfluenceMaximumTerm1952}).

\bibliographystyle{alpha}
\bibliography{Burgers}

\end{document}